\documentclass[english,11pt]{amsart}
\usepackage{amsopn,amsthm,amsfonts,amsmath,amssymb,amscd}
\usepackage{babel}
\usepackage{amssymb,tikz}

\newtheorem{Theorem}{Theorem}[section]
\newtheorem{Lemma}[Theorem]{Lemma}
\newtheorem{Proposition}[Theorem]{Proposition}
\newtheorem{Corollary}[Theorem]{Corollary}
\newtheorem{Example}[Theorem]{Example}

\newenvironment{Proof*}{{\it Proof.}}

\newcommand{\NN}{\mathbb{N}}

\newcommand{\ZZ}{\mathbb{Z}}

\newcommand{\BB}{\mathcal{B}}
\newcommand{\nn}{\mathcal{N}}
\newcommand{\mm}{\mathcal{M}}
\newcommand{\DEF}[1]{\emph{#1}}

\newcommand{\GF}[1]{\mathop{GF}({#1})}

\newcommand{\Ann}{{\rm Ann}}

\newcommand{\diam}[1]{{\rm diam}(#1)}
\newcommand{\girth}[1]{{\rm girth}(#1)}

\begin{document}

\title{The zero--divisor graphs of rings and semirings}
\author{David Dol\v zan, Polona Oblak}
\date{\today}

\address{D.~Dol\v zan:~Department of Mathematics, Faculty of Mathematics
and Physics, University of Ljubljana, Jadranska 19, SI-1000 Ljubljana, Slovenia; e-mail: 
david.dolzan@fmf.uni-lj.si; tel: +38614766624; fax: +38614766600}
\address{P.~Oblak: Faculty of Computer and Information Science,
Tr\v za\v ska 25, SI-1000 Ljubljana, Slovenia; e-mail: polona.oblak@fri.uni-lj.si}

\bigskip

\begin{abstract} 
In this paper we study  zero--divisor graphs of rings and semirings.
We show that all zero--divisor graphs of (possibly noncommutative) semirings are connected and have 
diameter less than or equal to 3. We characterize all acyclic zero--divisor graphs of semirings and
prove that in the case  zero--divisor graphs are cyclic, their  girths are less than or equal to 4. 
We find all possible cyclic zero--divisor graphs over commutative semirings 
having at most one 3-cycle, and characterize all complete $k$-partite and regular zero--divisor graphs.
 Moreover, we characterize all additively cancellative commutative
semirings and all commutative rings such that their zero--divisor graph has exactly one 3-cycle.
\end{abstract}

\maketitle 

\section{Introduction}

\bigskip


For any semigroup $S$ with zero, we denote by $Z(S)$ the set of zero--divisors, $Z(S)=\{x \in S;$ there exists
$0 \neq y \in S \text { such that } xy=0 \text { or } yx=0 \}$. 
We denote by $\Gamma(S)$ the \DEF{zero--divisor graph} of $S$. 
The vertex set $V(\Gamma(S))$ of $\Gamma(S)$ is the set of elements in $Z(S)^*=Z(S) \setminus \{0\}$ and 
an unordered pair of vertices $x,y \in V(\Gamma(S))$, $x \neq y$, is an edge 
$x-y$ in $\Gamma(S)$ if $xy = 0$ or $yx=0$.

Similarly, we can define the zero--divisor graphs of other algebraic structures, e.g. rings, semirings, near-rings, algebras.

\bigskip

The zero--divisor graphs of rings have been first introduced by Beck in \cite{beck88} in the study of 
commutative rings, and later studied by various authors, see for example
 \cite{akbari04, akbari06, akbari07, andliv99,andbad08, andmul07, bozpet09,
 chleewang10, lucas06, redmond02}.
The zero--divisor graphs are also intensely studied in the semigroup setting, e.g. \cite{dem02, demdem, dem10, dem10rm}.
Recently, they were used to study near-rings (see e.g.  \cite{cann05}) and
semirings (see e.g. \cite{atani08, atani09}).

\bigskip

In this paper we investigate the interplay between the algebraic properties of a (semi)ring and the graph theoretic properties of its zero--divisor graph.
In the next section, we give all necessary definitions.  In Section 3, we survey some of the known results of
the theory of the zero--divisor graphs over semigroups, rings, and semirings, and extend these results to a more general
setting of a noncommutative semiring and  we characterize all acyclic zero--divisor graphs of semirings (Theorem \ref{thm:characterization}). 
Next, we study the cyclic zero--divisor graphs. Firstly, we characterize the complete $k$-partite and regular zero--divisor
graphs that can appear as the zero--divisor graphs of commutative semirings (Theorem \ref{thm:kpartite} and 
Corollary \ref{thm:regular}). 
In the case the zero--divisor graph of a commutative semiring contains at most one triangle, we find 
all possible zero--divisor graphs (Theorems \ref{thm:characterizationg=4} and
 \ref{thm:completecharacterizationg=3}, Proposition \ref{thm:characterizationg=3}).  
 If the zero--divisor graph of a commutative semiring is cyclic and contains no triangles, we describe the order of the nilpotent elements in the semiring (Proposition \ref{thm:nilp}).
In the case the zero--divisor graph of an additively cancellative semiring contains exactly one triangle, we prove that the semiring has to be a ring 
(Proposition \ref{thm:ACsemi}) and we then proceed to characterize
all rings and their zero--divisor graphs containing exactly one triangle (Theorem \ref{thm:13-cycle}).

\bigskip
\bigskip

\section{Definitions}

\bigskip

A \emph{semiring} is a set $S$ equipped with binary operations $+$ and $\cdot$ such that $(S,+)$ is a commutative monoid with identity element 0, and $(S,\cdot)$ is a monoid with identity element 1. 
In addition, operations $+$ and $\cdot$ are connected by distributivity and 0 annihilates $S$. A semiring is 
\emph{commutative} if $ab=ba$ for all $a,b \in S$. 
 A semiring is \DEF{entire} (or \DEF{zero--divisor-free})  if $ab=0$ implies that $a=0$ or $b=0$. 
The semiring $S$ is \emph{additively cancellative} if $a+c=b+c$ implies that $a=b$ for all $a,b,c \in S$.

The simplest example of a commutative semiring  is the \DEF{binary Boolean semiring}, the set $\{0,1\}$
in which $1+1=1\cdot1=1$. We denote the binary Boolean semiring by $\BB$.
Moreover, the set of nonnegative integers (or reals) with the usual operations of addition and multiplication,
is a commutative  semiring. Other examples of commutative semirings are distributive lattices, tropical 
semirings etc.

\bigskip

The sequence of edges $x_0 - x_1$, $ x_1 - x_2$, ..., $x_{k-1} - x_{k}$ in a graph is called \emph{a path of length $k$} and is 
denoted by $x_0 - x_1 - \ldots - x_k$ or $P_{k+1}$. The \DEF{distance} between two vertices is the length of the shortest
path between them. The \DEF{diameter} $\diam{\Gamma}$ of the graph $\Gamma$ is the longest
 distance between any two vertices of the graph.
A path $x_0 - x_1 - \ldots - x_{k-1} - x_0$ is called a \emph{cycle}.  The \DEF{girth} of the graph $\Gamma$
is the length of the shortest cycle contained in the graph  and will be denoted by $\girth{\Gamma}$.

\bigskip

The complete 
graph  will be denoted by $K_n$ and complete bipartite graph by $K_{m,n}.$
We say that the \emph{star graph} is a complete bipartite graph $K_{1,n}$.  Note that $K_{1,0}=K_1$.
The \emph{two-star graph} $S_{m,n}$, where $n ,m \in {\mathbb N}\cup\{0\}$, is a graph with the set of 
vertices  equal to
the set $\{ v_1,v_2, u_1,u_2,\ldots,u_m,w_1,w_2,\ldots,w_n\}$, and with the following edges: $v_1 - v_2$, $u_i - v_1$ 
for $i=1,2,\ldots,m$, and $w_j - v_2$ for $j=1,2,\ldots,n$. Note that $S_{0,n}=K_{1,n+1}$ is a star graph.
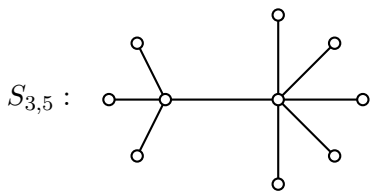
\begin{figure}[h]
	\centering
	\begin{tikzpicture}[style=thick, scale=0.75]
	\draw (-5,0) node[anchor=west]{$S_{3,5}:$};
		\draw (0,1.5) -- (0,0) -- (1,1);
		\draw  (1.5,0) -- (0,0);
		\draw  (1,-1) -- (0,0);
		\draw  (0,-1.5) -- (0,0) -- (-2,0) -- (-3,0);				
		\draw  (-2.5,1) -- (-2,0) -- (-2.5,-1);
		\draw[fill=white] (0,1.5) circle (1mm)  (0,0) circle (1mm) (1,1) circle (1mm) (1.5,0) circle (1mm) (1,-1) circle (1mm) (0,-1.5) circle (1mm)  (-2,0) circle (1mm) (-3,0) circle (1mm) (-2.5,1) circle (1mm) (-2.5,-1) circle (1mm);
		\end{tikzpicture}
	\end{figure}


Let $\overline{K}^{\; r}_{m,n}$ be the complete bipartite graph $K_{m,n}$ together with $r$ 
 vertices $v_1,v_2,\ldots,v_r$ and edges $v_i-a$ for all $i$ and some vertex $a$, 
 such that $\deg(a)=n$ in the induced subgraph $K_{m,n}$.
Moreover, choose vertex $b$, such that $\deg(b)=m$ and $a-b$ is an edge in $\overline{K}^{\; r}_{m,n}$. 
Denote by  $K^{\;  \triangle(r_1,r_2,r_3)}_{m,n}$ 
the graph   $\overline{K}^{\; r_1}_{m,n}$ together with vertices $e$, $v_i$, $w_j$
and edges $a-e-b$,  $b-v_i$, $e-w_j$, $i=1,2,\ldots,r_2$, $j=1,2,\ldots,r_3$.

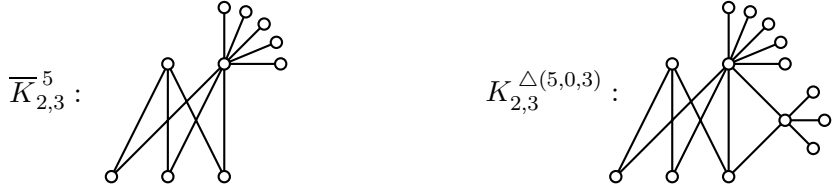
\begin{figure}[h]
	\centering
	\begin{tikzpicture}[style=thick, scale=0.75]
	\draw (-4,-0.5) node[anchor=west]{$\overline{K}^{\; 5}_{2,3}:$};
	\draw \foreach \x in {0,22.5,45,67.5,90} {
				(0,0) node{} -- (\x:1) node{}
			};
 		\path (-1,0) coordinate(c);
		\path (0,-2) coordinate(b);
		\path (-1,-2) coordinate(d1);
		\path (-2,-2) coordinate(d2);
		\draw \foreach \x in {(b),(d1),(d2)} {
				(0,0)  -- \x
				(c)  -- \x
			};
		\draw[fill=white] (c) circle (1mm) (b) circle (1mm) (d1) circle (1mm) (d2) circle (1mm) ;					\draw[fill=white] (0,0) circle (1mm) (0:1) circle (1mm) (45:1) circle (1mm) (90:1) circle (1mm) (67.5:1) circle (1mm) (22.5:1) circle (1mm);
\end{tikzpicture}
		\qquad \qquad\qquad
\begin{tikzpicture}[style=thick, scale=0.75]
	\draw (-4.5,-0.5) node[anchor=west]{$K^{\;  \triangle(5,0,3)}_{2,3}:$};
		\draw \foreach \x in {0,22.5,45,67.5,90} {
				(0,0) node{} -- (\x:1) node{}
			};
		\path (-1,0) coordinate(c);
		\path (0,-2) coordinate(b);
		\path (-1,-2) coordinate(d1);
		\path (-2,-2) coordinate(d2);
		\draw \foreach \x in {(b),(d1),(d2)} {
				(0,0)  -- \x
				(c)  -- \x
			};
		\draw  (0,0) -- (1,-1) -- (b);				
		\draw  (1.5,-0.5)  -- (1,-1) -- (1.5,-1.5);	
		\draw  (1.7,-1)  -- (1,-1);	
 		\draw[fill=white] (0,0) circle (1mm) (0:1) circle (1mm) (45:1) circle (1mm) (90:1) circle (1mm) (67.5:1) circle (1mm) (22.5:1) circle (1mm);
		\draw[fill=white] (c) circle (1mm) (b) circle (1mm) (d1) circle (1mm) (d2) circle (1mm) (1,-1) circle (1mm) (1.5,-0.5) circle (1mm) (1.5,-1.5) circle (1mm) (1.7,-1) circle (1mm);	
		
				\end{tikzpicture}
	\end{figure}
	
%


%

\bigskip
\bigskip

\section{The zero--divisor graph of a semiring}

\bigskip

Let us investigate the zero--divisor graph of an arbitrary  (possibly noncommutative) semiring.

Firstly, we shall prove that the zero--divisor graphs of (noncommutative) semirings are always connected and 
have diameters at most $3$. 
This is a similar result to \cite[Thm. 3.2]{redmond02} (for rings) and \cite[Lemma 2.1]{atani08}
(for commutative semirings). 

\bigskip

\begin{Theorem}\label{thm:diam3}
 If $S$ is a semiring, then $\diam{\Gamma(S)}\leq 3$.
\end{Theorem}

\medskip

\begin{proof}
 Take $x,y \in Z(S)^*$, such that $xy \ne 0$ and $yx \ne 0$. We want to show that there is a path 
 from $x$ to  $y$  and $d(x,y)\leq 3$.
 
 There exist $a,b \in Z(S)^*$, such that $ax=0$ or $xa=0$ and $by=0$ or $yb=0$.
 Note that here, $a$ can be equal to $x$, as well as $b$ equal to $y$.
 If $a=b$, $ab=0$, $ba=0$, $ay=0$, $ya=0$, $bx=0$ or $xb=0$, then
 $d(x,y)\leq 3$. So, suppose, none of the above is true. In the case $ax=0$, we have that either
 $x-ba-y$ or $x-ya-b-y$ is a path joining $x$ and $y$. Otherwise, if $xa=0$,  either
 $x-ab-y$ or $x-ay-b-y$ is a path from $x$ to $y$. All of these four paths  are of length at most 3, even if
 some of the vertices coincide, and therefore $d(x,y)\leq 3$. 
\end{proof}

\bigskip

 Anderson and Mulay  \cite[Thm.~2.8]{andmul07}  proved that 
 for direct products of integral domains and their subrings,  the diameter is at most 2. We generalize
 this result to noncommutative entire semirings.
  
  \bigskip

\begin{Proposition}\label{thm:SxS}
 If $S_1$ and $S_2$ are entire semirings and $S \subseteq S_1 \times S_2$ is a semiring.
 If $\Gamma(S)\ne \emptyset$, then  
 $\diam{\Gamma(S)}\leq 2$.
\end{Proposition}

\medskip

\begin{proof}
 If  $S \subseteq S_1 \times S_2$, where $S_1$ and $S_2$ are entire semirings, then 
 $(s_1,s_2) \in Z(S)^*$ implies that either $s_1=0$ or $s_2=0$.
 
 Assume $\diam{\Gamma(S)}\geq 2$. Then, there exist $x,y \in Z(S)^*$, such that $xy \ne0$, $yx\ne 0$.
 Without loss of generality, let us assume that $x=(x_1,0)$. This implies that $y=(y_1,0)$. 
 Since $\Gamma(S)$ is a connected graph, there exists an edge $x-z$ in $\Gamma(S)$ and 
 $z=(0,z_1)$. Thus $x-y-z$ is a path in $\Gamma(S)$ and $\diam{\Gamma(S)}\leq 2$.
\end{proof}

\bigskip

In the following examples we show that some families of the graphs with $\diam \Gamma \leq 2$
can be realized as the zero--divisor graphs of semirings.
We will later need the realization of these families of graphs in the characterization of complete $k$-partite and
regular zero--divisor graphs.

\medskip

We will denote by $\mm_n(S)$ the set of all $n \times n$ matrices over a semiring $S$.
The matrix with the only nonzero entry 1 in the $i$th row and $j$th column will be denoted by $E_{i,j}$. 
The matrix $I_n$ will denote the $n \times n$ identity matrix, $0_n$ will denote the $n \times n$ zero matrix 
and $J_n$ will denote the matrix $E_{1,2}+E_{2,3}+\ldots+E_{n-1,n}$.
Also, let us denote by $A \oplus B$ the direct sum of matrix blocks $\left[\begin{matrix}
A&0\\
0&B
\end{matrix}
\right]$.

\bigskip
\begin{Example}\label{ex:Kmn}
We can realize all complete bipartite graphs as zero--divisor graphs of a direct product of two
semirings.   Namely, if $S$ and $T$ are arbitrary entire semirings  with $\vert S \vert = m+1$ and 
$\vert T \vert=n+1$,  then $\Gamma(S\times T)=K_{n,m}$. Such $S$ and $T$ exist, for example we can 
choose totally ordered  sets of appropriate cardinality.
 \hfill$\square$
\end{Example}

\bigskip
\begin{Example}\label{ex:Kn}
Anderson and Livingston \cite[Thm.~2.8]{andliv99} proved that  if the zero--divisor graph of a commutative ring 
$R$ with 1 is equal to $\Gamma(R)=K_n$, $n \geq 3$, then all zero--divisors are nilpotents 
of order 2. This statement does not hold for semirings.  However, we can show that if $S$ is a commutative 
semiring and $\Gamma(S)=K_n$, $n\geq 3$, then $x^2=0$ for all but possibly one $x \in Z(S)^*$.

Suppose $Z(S)^*=\{a_1,a_2\ldots,a_n\}$ and $a_1^2\ne 0$. Since $\Gamma(S)=K_n$, it follows that 
$a_ia_j=0$ for all $i\ne j$, and therefore $a_1(a_1+a_i)=a_1^2\ne 0$ for $i\ne 1$.
Since $a_j(a_1+a_i)=0$ for all $j \ne i$, $j,i \ne 1$, it follows that $a_1+a_i \in Z(S)^*$, thus
$a_1+a_i=a_1$ for all $i\ne 1$. By multiplying this equation by $a_i$, we have $a_i^2=0$ for 
all $i \ne 1$.

Moreover, such semirings $S$ indeed exist. Consider a semiring $S$ in $\mm_{2n-1}(\BB)$, generated by 
$a_i=J^{n-2+i}+J^{n-1+i}+\ldots+J^{2n-2}$, where $i=1,2,\ldots,n$. Since $J^{2n-1}=0$,
we have $a_i a_j =0$ for all $i,j=1,2,\ldots,n$ but for $i=j=1$. Therefore $\Gamma(S)=K_n$,
$a_i^2=0$ for all $i \ne 1$ and $a_1^2\ne 0$.
 \hfill$\square$
\end{Example}
\bigskip

The next two examples show that we can also realize all possible star and two-star graphs as the zero--divisor graphs of 
(even commutative) semirings.
Compare this with \cite[Ex. 2.1]{andliv99} where it has been shown that for a commutative ring, the zero--divisor graph cannot 
be equal to $P_4=S_{1,1}$.

\bigskip
\begin{Example}\label{ex:K1n}
 Let $\mm_{n+1}(\BB)$ be the semiring of $n+1$ by $n+1$ matrices over the Boolean semiring, where $n \geq 1$,
  and denote by  $S$ the subsemiring generated by the set 
  $\{I_1 \oplus 0_{n}, 0_1 \oplus I_{n}, 0_1 \oplus I_{n} + J_{n}\}$.  
  The zero--divisors in the semiring $S$ are of two types, 
 $I_1 \oplus 0_{n}$ and $0_1 \oplus (I_{n} + J_{n} + J_{n}^2 + \ldots + J_{n}^k)$ for $k=0,1, \ldots n-1$.
 It can be easily verified that then only the products of the element $I_1 \oplus 0_{n}$ 
 with $0_1 \oplus (I_{n} + J_{n} + J_{n}^2 + \ldots + J_{n}^k)$ 
 are equal to zero for all $k$, 
 so $\Gamma(S)=K_{1,n}=S_{0,n-1}$.
 
 Obviously, we can realize the graph $K_1=K_{1,0}$ as the zero--divisor graph of a semiring, for example the (semi)ring ${\mathbb Z}_4$.
 \hfill$\square$
\end{Example}

\bigskip

\begin{Example}\label{ex:K1nm}
 Choose $n, m \in {\mathbb N}$.
 Let $L=\{0,x_1,x_2,x_3,\ldots, 1\}$
 be any totally ordered (distributive) lattice containing at least $\max\{n,m\}$ nonzero elements.  
 Then $L$ is also an entire semiring
 for the operations $x_i + x_j=x_{\max\{i,j\}}$ and $x_i \cdot x_j=x_i x_j=x_{\min\{i,j\}}$.  
 
 Now, let $\mm_4(L)$ denote the semiring of all
 $4 \times 4$ matrices over $L$.  
 Denote $v_1=0_2 \oplus x_1J_2$ and $v_2=(x_1(I_2+J_2)) \oplus 0_2$.
 For $i=1,2,\ldots,n$ define 
 $u_i=(x_1I_2+x_iJ_2) \oplus x_1J_2$, and
 for $j=1,2,\ldots,m$ define
 $w_j=0_2 \oplus (x_j(I_2+J_2))$.
 Let $S$ denote the subsemiring of $\mm_4(L)$, generated by the elements $v_1,v_2,u_1,u_2,\ldots, u_n, w_1,w_2, \ldots, w_m$.
 
 It can be easily seen, that since $L$ is entire and antinegative, we do not get any zero--divisors in $S$ that are not
 already amongst the generating elements. So, the zero--divisor graph of $S$ consists of edges $v_1 - v_2$, 
 $u_i - v_1$ for $i=1,2,\ldots,n$, and $w_j - v_2$ for $j=1,2,\ldots,m$, which implies that
 $\Gamma(S)=S_{n,m}$ is a two-star graph. \hfill$\square$
\end{Example}

\bigskip

We shall now see, that we can consider the case of cyclic zero--divisor graphs separately from the case of
acyclic ones. 
We will find all possible acyclic graphs that can be realized as zero--divisor graphs of semirings, 
and for the cyclic graphs, we shall prove that they always contain at least one cycle of length at most $4$.

\bigskip
\begin{Lemma}\label{thm:path5}
If $P_5$ is a subgraph of $\Gamma(S)$, where $S$ is an arbitrary semiring, then $\Gamma(S)$ is a 
cyclic graph and $\girth{\Gamma(S)} \leq 4$.
\end{Lemma}

\medskip

\begin{proof}
 Denote by $a-b-c-d-e$ the path $P_5$ in $\Gamma(S)$. Suppose that $\girth{\Gamma(S)} > 4$,
 i.e., there are no edges among other vertices from this path. 
 
 Consider first the case $ba=bc=0$. Since $eb \ne 0$ and $(eb)a=(eb)c=0$, we have that $eb=b$. 
 (Otherwise, there is a cycle of length 3 or 4 in  $\Gamma(S)$.) Similarly, we conclude that $db=b$.
 Since $d-e$ is an edge in  $\Gamma(S)$, we have either that $de=0$, and thus $db=dbe=0$, or
 $ed=0$, and therefore $eb=edb=0$, which both contradict the asumption that  $\girth{\Gamma(S)} > 4$.
 
 Similarly, we can treat the case $ab=cb=0$.
 
 Suppose now $ab=bc=0$ and $ba\ne 0$, $cb \ne 0$. 
 By Theorem \ref{thm:diam3}, we have that $d(a,e)\leq 3$. Since
 $\girth{\Gamma(S)} > 4$, the path from $a$ to $e$ of the length at most 3 cannot contain any of 
 vertices $b$, $c$, $d$. 
 If $d(a,e)=3$ and $a-x-y-e$ is a path from $a$ to $e$, we obtain a cycle 
 $a-b-c-d-e-y-x-a$ of length 7. Note that if $d(a,e)=2$, then we can assume that $y=x$ and if $d(a,e)=1$,
 then $e=y=x$. In all three cases, let us assume, that $a-x$ is an edge in $\Gamma(S)$.
 If we assumed $ax=0$, we would get a contradiction as in the case $ba=bc=0$.
 Thus, from now, let $xa=0$. 
 Since $\girth{\Gamma(S)} > 4$ and $b(cx)=0$, the product $cx$ is either equal to $a$, $b$, $c$
 or is an element, different from $a,b,c,d,e,f,x,y$. In the first case, $a^2=(cx)a=0$ and therefore
 $a-ac-x-a$ is a 3-cycle in $\Gamma(S)$, a contradiction. In the second case, $ba=(cx)a=0$, which is 
 again a contradiction.  Otherwise, $b-cx-a-b$ is a cycle of length 3. 
\end{proof}

\bigskip
\begin{Corollary}\label{thm:cycle5}
The cycle on $n$ vertices, $n \geq 5$, cannot be realized as $\Gamma(S)$, where $S$ is a semiring.
\end{Corollary}
\bigskip

We can now prove the theorem that generalizes \cite[Thm.~2.4, Thm.~2.5]{andmul07} and provides a characterization
of all acyclic zero--divisor graphs over semirings.

\bigskip

\begin{Theorem}\label{thm:characterization}
  Let $S$ be a non--entire semiring.  
  \begin{enumerate}
  \item[(a.)] If $\Gamma(S)$ is a cyclic graph,
      then $\diam{\Gamma(S)}\leq 3$ and $\girth{\Gamma(S)}\leq 4$.
  \item[(b.)] $\Gamma$ is an acyclic zero--divisor graph of  a non--entire semiring if and only if
      $\Gamma=S_{n,m}$ or $\Gamma=K_{1}$.
 \end{enumerate} 
\end{Theorem}

\medskip

\begin{proof}
If $\Gamma(S)$ is a cyclic graph which contains a cycle of length $5$ or more, then it also contains $P_5$.
By Lemma \ref{thm:path5}, it follows $\girth{\Gamma(S)}\leq 4$.  Note that $\diam{\Gamma(S)}\leq 3$ by Theorem \ref{thm:diam3}.
Assume now that $\Gamma(S)$ is acyclic and contains at least 2 vertices.  
Again by Lemma \ref{thm:path5}, we know that it does not contain
$P_5$, so the only possibility is that  $\Gamma(S)=S_{n,m}$ for some $n,m$.
The converse of (b.) follows from Examples \ref{ex:K1n} and \ref{ex:K1nm}.
\end{proof}

\bigskip

This result characterizes the acyclic zero--divisor graphs of (non)-\-commutative
semirings. In the following sections we will examine the cyclic zero--divisor graphs of commutative semirings.

\bigskip
\bigskip

\section{The complete $k$-partite and regular zero--divisor graphs of commutative semirings}

\bigskip

In this section we investigate two special families of cyclic graphs, complete $k$-partite and regular graphs.
 DeMeyer et al. \cite{dem10} showed that all complete 
$k$-partite graphs are zero--divisor graphs of commutative semigroups, and (see Theorem \ref{thm:dem10})  characterized all
 regular graphs that can appear as the zero--divisor graphs of commutative semigroups. In the semiring setting, these two
 assertions no longer hold, and  in Theorem \ref{thm:kpartite} and Corollary \ref{thm:regular} we shall characterize  complete
  $k$-partite and regular graphs that can appear as the zero--divisor graphs of commutative semirings.

\bigskip

\begin{Theorem}[DeMeyer, Greve, Sabbaghi, Wang \cite{dem10}] \label{thm:dem10}
 Let $\Gamma$ be a connected $k$-regular graph with $n$ vertices.
 Then $\Gamma$ is a zero--divisor graph of a commutative semigroup
  if and only if $n-k \vert n$ and  $\Gamma =\bigvee^{n/(n-k)} (n-k) \, K_1$.
\end{Theorem}

\bigskip

\begin{Theorem}\label{thm:kpartite}
 Let $\Gamma$ be a complete $k$-partite graph with $n$ vertices and $k\geq 2$.
 Then $\Gamma$ is a zero--divisor graph of a commutative semiring
  if and only if $k=2$ or $\Gamma=K_{k-1} \bigvee (n-k+1) \, K_1$.
\end{Theorem}
\medskip
\begin{proof}
 Since $\Gamma$ is connected, we have $k\geq 2$. 
 Suppose $\Gamma=C_1 \bigvee C_2 \bigvee \ldots \bigvee C_k$ is a complete $k$-partite zero--divisor  graph with $k\geq 3$. If  $\Gamma\ne K_{k-1} \bigvee (n-k+1) \, K_1$, 
 then there exist, say $C_1$ and $C_2$ with $\vert C_1\vert \geq 2$ and  $\vert C_2 \vert \geq 2$.  Let $a_1,b_1\in C_1$, $a_2,b_2 \in C_2$ and $c\in C_3$. Since 
 $a_1c=a_2c=0$, it follows that $(a_1+a_2)c=0$, so $a_1+a_2 \in Z(S)$. If $a_1+a_2 \notin C_1$, then $(a_1+a_2)b_1=0$ and thus $a_1b_1=0$, a contradiction. Therefore, 
 $a_1+a_2 \in C_1$ and similarly we obtain  $a_1+a_2 \in C_2$ which is also a contradiction. 
 
  Example \ref{ex:Kmn} 
  shows that $K_{m,n}$ 
  can be realized as the zero--divisor graph of a commutative semiring.
  Choose an integer $k$, $3 \leq k \leq n-1$ and consider the subsemiring $S \subseteq \mm_{2n+1}(\BB)$, generated by 
  matrices $\{A_i, B, C_j; \; 2\leq i \leq n-k+1, \, 2n-k+2\leq j \leq 2n\}$,
  where $J=J_{2n+1}$, $B=J^{n}+J^{n+1}+\ldots+J^{2n}$, $C_j=J^{j}+J^{j+1}+\ldots+J^{2n}$ and
  $A_{i}=[a^i_{st}]$ are the matrices with entries
  $$a^i_{st}=\begin{cases}
   0, &t-s<n \, \text{ or } \, s=i,\,  t=i+n,\\
   1, & \text{otherwise.}
  \end{cases}$$
  Observe that $S$ is a semiring with $\Gamma(S)=\{C_1,C_2,\ldots,C_k\} \bigvee \{A_1,A_2,\ldots,A_{n-k+1},B\}=K_{k-1} \bigvee (n-k+1) \, K_1$.
\end{proof}

\bigskip

\begin{Corollary}\label{thm:regular}
 Let $\Gamma$ be a $r$-regular graph with $n$ vertices.
 Then $\Gamma$ is a zero--divisor graph of a commutative semiring if and only if 
 $r=n-1$ and $\Gamma=K_n$, or 
 $n$ even,  $r=\frac{n}{2}$ and $\Gamma=K_{\frac{n}{2},\frac{n}{2}}$.
\end{Corollary}
\medskip
\begin{proof}
 Assume that $\Gamma$ is a zero--divisor graph of a semiring. Then, $\Gamma$ is
 connected by Theorem \ref{thm:diam3} and thus by \ref{thm:dem10}, $\Gamma$ is a join 
 of $\frac{n}{n-r}$ copies of $(n-r)\, K_1$, which is a complete $\frac{n}{n-r}$-partite graph.
 Now, Theorem \ref{thm:kpartite} implies that there are two possibilities.
 In the first case, $\frac{n}{n-r}=2$ and thus $\Gamma$ is a $r$-regular bipartite graph with 
 $r=\frac{n}{2}$, so $\Gamma=K_{\frac{n}{2},\frac{n}{2}}$.
 In the second case, $\Gamma=K_{\frac{r}{n-r}} \bigvee (n-\frac{r}{n-r}) \, K_1$. Since $\Gamma$
 is $r$-regular, it follows that $n-1=\frac{r}{n-r}$ and therefore $r=n-1$, so $\Gamma=K_n$.
 
 Examples \ref{ex:Kmn} and \ref{ex:Kn} show that  $K_n$ and $K_{\frac{n}{2},\frac{n}{2}}$
 can both be realized as the zero--divisor graphs of commutative semirings.
\end{proof}

\bigskip
\bigskip

\section{The cyclic zero--divisor graphs}

\bigskip

In this section we will study the cyclic zero--divisor graphs of commutative semirings.  By Theorem 
\ref{thm:characterization}  every cyclic zero--divisor graph has a 3-cycle or a 4-cycle.
 We will define the following graphs, which we shall prove are the graphs that cannot appear as
the induced subgraphs of a cyclic zero--divisor graph of a commutative semiring.
\begin{itemize}
  \item $C_{4,4}$, which is a graph consisting of two cycles  $a-b-c-d-a$ and $a-b-f-e-a$ with the common edge 
  $a-b$, and
 \item $C_4'$, which is a 4-cycle $a-b-c-d-a$ together with two vertices $e$ and $f$, connected
   with edges $a-e$ and $b-f$,  
 \item $C_4^{''}$, which is a 4-cycle $a-b-c-d-a$ together with two vertices $e$ and  
   $f$, connected with edges $a-e$ and $c-f$,
\end{itemize}

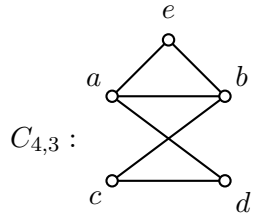
		\begin{figure}[h]
			\centering
			\begin{tikzpicture}[style=thick, scale=0.75]
				\draw (1,1) -- (-1,-0.5) -- (1,-0.5)  -- (-1,1) -- (1,1) -- (0,2) -- (-1,1);
				\draw[fill=white] (0,2) circle (1mm)  (-1,1) circle (1mm) (-1,-0.5) circle (1mm) (1,-0.5) circle (1mm) (1,1) circle (1mm);
				\draw (-1,1) node[anchor=south east]{$a$};
				\draw (1,1) node[anchor=south west]{$b$};
				\draw (-1,-0.5) node[anchor=north east]{$c$};
				\draw (1,-0.5) node[anchor=north west]{$d$};
				\draw (0,2.2) node[anchor=south]{$e$};
				\draw (-3,0.2) node[anchor=west]{$C_{4,3}:$};
			\end{tikzpicture}
		\end{figure}

Moreover, let us define the graph $C_{4,3}$, which is a graph consisting of a 4-cycle  $a-b-c-d-a$ and a
3-cycle $a-b-e-a$ with the common edge $a-b$.

\bigskip

We can now state the following lemma.
\bigskip

\begin{Lemma}\label{thm:girth=3}
  Let $S$ be a commutative semiring and let the graph $\Gamma(S)$  contain exactly one
  $3$-cycle and at least one $n$-cycle, $n \geq 4$. Then, $\Gamma(S)$ contains $C_{4,3}$
  as an induced subgraph.  
\end{Lemma}

\medskip
\begin{proof}
Let $n\ge 4$ be the smallest integer, such that $\Gamma(S)$ contains an $n$-cycle 
$x_1-x_2-x_3-\ldots-x_n-x_1$. If $x_1x_3= 0$, we obtain a $n-1$ cycle in the graph 
$\Gamma(S)$ and thus $n=4$ and $\Gamma$ has two 3-cycles, a contradiction.
Since $(x_1x_3)x_2=(x_1x_3)x_4=(x_1x_3)x_n=0$, either $x_1x_3$ is a vertex on the cycle
or $x_1x_3\ne x_i$ for all $i$. In both cases $\Gamma(S)$ contains a 4-cycle.

 Suppose $\Gamma(S)$ contains a $4$-cycle $a-b-c-d-a$ and a 3-cycle $e-f-g-e$.
 We shall firstly prove that they have a common vertex. 

 Choose an arbitrary vertex in the 3-cycle, say $e$. If $e$  is a neighbour of at least 2 vertices from the
 4-cycle (say, one of them is $a$), then note that either $\Gamma(S)$ contains more than one 3-cycle (which contradicts the assumption),  or the only other neighbour of $e$ in the 4-cycle is $c$. In the latter case 
 $\Gamma(S)$ contains a 4-cycle (either $a-e-c-b-a$ or $a-e-c-d-a$) and the 3-cycle $e-f-g-e$ with the 
 common vertex $e$.
 So, suppose every vertex in the 3-cycle $e-f-g-e$ has at most one neighbour in the 4-cycle. In this case, there
 exists a vertex in the 4-cycle, say $a$, such that $ae\ne 0$. Since $(ae)f=(ae)g=0$ and $\Gamma(S)$ has
 only one 3-cycle, it follows that $ae$ is an element of $\{e,f,g\}$. On the other hand, $(ae)b=(ae)d=0$,
 so $ae$ has at least 2 neighbours in the 4-cycle. It follows that $ae=c$ and 4-cycle and 3-cycle have a
 common vertex. 
 
 We proved  that 3-cycle and 4-cycle have a common vertex, for instance $d=g$. If $a = e$, then
 the Lemma is proven. Otherwise, since the graph contains only one 3-cycle,  $ae\ne 0$ and 
 $(ae)d=(ae)f= 0$ imples that $ae$ is an element of $\{d,e,f\}$. Moreover, $(ae)b=0$ and thus $\Gamma(S)$
 contains $C_{4,3}$, since $\Gamma(S)$ may contain only one 3-cycle. 
\end{proof}

\bigskip

\begin{Lemma}\label{thm:induced=4}
  Let $S$ be a commutative semiring with $\girth{\Gamma(S)}=4$. Then $C_4'$  cannot appear as an induced   subgraph of $\Gamma(S)$.
\end{Lemma}

\medskip
\begin{proof}
 Suppose that $\Gamma(S)$  contains $C_4'$, which is a 4-cycle $a-b-c-d-a$ together with two vertices $e$ 
   and $f$, connected with edges $a-e$ and $b-f$. 
   Firstly,  $ef\ne0$ and $(ef)a=(ef)b=0$. Again, since $\girth{\Gamma(S)}> 3$, it follows that either 
   $ef=a$ or $ef=b$. By the symmetry, we can assume that $ef=a$. 
   Now, $e(fd)=(ef)d=0$ and moreover $a(fd)=b(fd)=c(fd)=0$. Since $fd \ne 0$ and
    $\girth{\Gamma(S)}> 3$,  the product $fd$ cannot exist as a vertex  in $\Gamma(S)$.
 \end{proof}
\bigskip

\begin{Lemma}\label{thm:induced=3}
  Let $S$ be a commutative semiring with $\Gamma(S)$ containing at most one
  3-cycle. Then neither  $C''_4$,
   $C_{4,4}$, nor $C_{4,5}$  can appear as induced subgraphs of $\Gamma(S)$.
\end{Lemma}

\medskip
\begin{proof}
 Suppose $\Gamma(S)$ contains $C_{4,4}$ as an induced subgraph, i.e. $\Gamma(S)$ 
 contains vertices
    $a,b,c,d,e,f$, where the only edges are  $a-b-c-d-a$ and $a-b-f-e-a$.
    Consider the product $ce$. Clearly, $(ce)a=(ce)b=(ce)d=(ce)f=0$ and $ce \ne 0$, and since 
     $\Gamma(S)$ contains at most one
  3-cycle, such vertex $ce$ cannot exist in $\Gamma(S)$.

If $\Gamma(S)$ contains $C_{4,5}$ as an induced subgraph, i.e. $\Gamma(S)$ contains vertices
    $a,b,c,d,e,f$, where the only edges are  $a-b-c-d-a$ and $a-b-c-f-e-a$, then similarly as in (1),  $ce \ne 0$ is a 
    zero divisor in $S$, but  cannot exist as a vertex  in $\Gamma(S)$.


If $\Gamma(S)$  contains $C_4^{''}$, which is a 4-cycle $a-b-c-d-a$ together with two vertices 
 $e$ and  $f$, connected with edges $a-e$ and $c-f$, note that $ce\ne 0$ and $(ce)a=(ce)b=(ce)d=(ce)f=0$. Since 
      $\Gamma(S)$ contains at most one
  3-cycle, it follows that $ce$ cannot exist as a vertex  in $\Gamma(S)$.
\end{proof}

\bigskip
\bigskip

\section{Commutative semirings with zero--divisor graphs of girth 4}

\bigskip

In this section we shall describe the zero--divisor graphs of commutative semirings with their
girth equal to 4. 

If the semiring is  a ring, the structure of the ring itself can be deduced from the properties
of its zero--divisor graph. Anderson and Mulay, \cite[Theorems 2.3 and 2.4]{andmul07} have characterized commutative 
rings $R$ with  $\girth{\Gamma(R)}=4$. Their findings about this can be summarized in the following theorem.

\bigskip

\begin{Theorem}[Anderson, Mulay, \cite{andmul07}]
    If $R$ is a commutative ring with identity such that $\girth{\Gamma(R)}=4$, then
    \begin{enumerate}
    \item either $\Gamma(R)=K_{m,n}$, $m,n \geq 2$ and the total quotient ring of $R$ is a direct product of two 
    fields $F_1 \times F_2$, $\vert F_i \vert \geq 3$, 
    \item or $\Gamma(R)=\overline{K}^{\; m}_{3,m}$ 
    and $R=D \times B$, where $D$ is an integral domain with at least 3 elements and
    $B \in \{ \ZZ_4,  \ZZ_2[X]/(X^2) \}$.
%
    \end{enumerate}
\end{Theorem}

\bigskip

 The following examples
 show that there exist large families of commutative semirings with their zero--divisor 
 graphs equal to  $\overline{K}^{\; r}_{n,m}$.

\bigskip
 
 \begin{Example}\label{ex:overlineK}
  Let $S=\{0,a_1,a_2,\ldots,a_{m-1},1\}$ be a totally ordered lattice and $T \subseteq \mm_{n-1}(\BB)$
  the commutative semiring, generated by   $A=J_{n-1}^{n-2}$ and $B=I_{n-1}+J_{n-1}$.
  Note that  $Z(S)=\{0\}$ and $Z(T)=\{0,A\}$. Then, $Z(S\times T)$ consists of two types of
   elements $(s,t)\in S\times T$: the first type are those having  $s=0$ or $t=0$, which form 
   the induced subgraph $K_{n,m}$ of $\Gamma(S\times T)$; the second type are the elements having
   $t=A$, which are all neighbours of the vertex $(0,A)$.  Thus, 
   $\Gamma(S\times T)= \overline{K}^{\; m}_{n,m}$.
   \hfill$\square$
 \end{Example}

\bigskip
\begin{Example}\label{ex:overlineK2}
   Let $S=\mm_2(\BB)$ and consider the semiring $S^{2n-3}$ for some $k\in \NN$. Denote by
  $e_i$ the element in  $S^{2n-3}$,   which has its only nonzero entry equal to $I_2$ in the $i$-th position
  and moreover let $I=I_2$ and $J=J_2$. Let us define the following  elements:
    \begin{center}
     \begin{tabular}{lcl}
        $a=J e_1$ &\quad \quad&$b=e_2$ \\
         $c_j=(I+J)e_1+\sum\limits_{i=1}^{j} J e_{2i+1}$ for $j=1,2,\ldots,n-2$ &\quad\quad& $d=e_2+e_4$\\
        $e=\sum\limits_{i=0}^{n-2} J e_{2i+1}+\sum\limits_{i=1}^{n-2}  e_{2i}$ &\quad\quad&
       \end{tabular}
    \end{center}
 Consider the semiring $T$ generated by $\{1,a,b,c_1,c_2,\ldots,c_{n-2},d,e\}$ and observe that  
   $\Gamma(T)= \overline{K}^{\; 3}_{n,2}$.
   \hfill$\square$
 \end{Example}

\bigskip

The following theorem shows that all zero--divisor graphs with their girth equal to 4 are actually of 
this form.

\bigskip

\begin{Theorem}\label{thm:characterizationg=4}
    If $S$ is a commutative semiring and  $\girth{\Gamma(S)}=4$ then
     $\Gamma(S)=\overline{K}^{\; r}_{m,n}$.
\end{Theorem}

\medskip

\begin{Proof*}
 Since $\girth{\Gamma(S)}=4$, $\Gamma(S)$ contains $K_{2,2}$ as induced subgraph.
 We proceed inductively by adding vertices while always maintaining 
 $\Gamma(S)=\overline{K}^{\; \rho}_{\mu,\nu}$ for some $\rho,\mu,\nu$.
 
 Assume that we have in $\Gamma(S)$ an induced subgraph $\overline{K}^{\; \rho}_{\mu,\nu}$,
 $\mu, \nu \geq 2$. 
 Let us decompose the vertex set of  $\overline{K}^{\; \rho}_{\mu,\nu}$ into 
 3 sets: $V_1=\{v; \; \deg{v}=1\}$ (possibly empty), $V_2$ and $V_3$ are the bipartite parts of
 $K_{\mu,\nu}$, where each  vertex in $V_3$ has degree $\mu$. Moreover, let $a\in V_2$ be
 the vertex with $\deg(a)=\nu+\rho$. If $\rho=0$, then choose $a$ to be any vertex in $V_2$.
 
 Choose any vertex $x \in \Gamma(S)$, that is not in $ \overline{K}^{\; \rho}_{\mu,\nu}$.
 \begin{itemize}
  \item If $x-v$ is an edge for some $v \in V_1$, then by Theorem \ref{thm:diam3}, $x-w$ is an edge 
  for some $w \in V_2\cup V_3$. If $w \in V_2$, then $\Gamma(S)$ contains $C_{4,5}$ as an induced
  subgraph, and if $w \in V_3$, then $\Gamma(S)$ contains $C_{4,4}$ as an induced
  subgraph. Both conclusions contradict Lemma \ref{thm:induced=3}.
 \item  If $\deg(x)=1$ and $x-a$ is an edge, or $V_1 = \emptyset$, then we get 
      $\overline{K}^{\; \rho+1}_{\mu,\nu}$.
 \item  If $V_1 \ne  \emptyset$ and $\deg(x)=1$, then if $x-v$ is an edge for some $v \in V_2\backslash \{a\}$, 
    then $\Gamma(S)$  contains $C''_4$ as an induced subgraph and otherwise, if $x-v$ is an edge for some 
    $v \in V_3$,  then $\Gamma(S)$  contains $C'_4$ as an induced subgraph, which  contradicts Lemma 
    \ref{thm:induced=4}.
 \item  If $\deg(x)\geq 2$ and $w-x-v$ is a path, then $v,w \in V_2$ or $v,w \in V_3$. (Otherwise, 
   $\girth{\Gamma(S)}=3$.)  Say, $v,w \in V_2$.  Suppose there exists $u \in V_2$ such that $xu \ne 0$.
   Now, $(xu)y=(xu)w=(xu)v=0$ for all $y \in V_3$, and this contradicts the assumption
   that  $\girth{\Gamma(S)}=4$. Thus, $x-u$ is an edge in $\Gamma(S)$ for all $u \in V_2$, so we get
   $ \overline{K}^{\; \rho}_{\mu,\nu+1}$.
   Similarly, if $v,w \in V_3$, we  get $ \overline{K}^{\; \rho}_{\mu+1,\nu}$. \hfill$\blacksquare$
 \end{itemize} 
\end{Proof*}

\bigskip

\bigskip

The next 
observation is a semiring generalization of a result that appears in \cite{andlevsha03} for
the ring theoretic case.

\bigskip

\begin{Proposition}\label{thm:nilp}
 If $S$ is a commutative semiring with $\girth{\Gamma(S)}=4$, then all nilpotents are of the order equal to 2.
\end{Proposition}

\medskip

\begin{proof}
 Note that since  $\girth{\Gamma(S)}=4$, graph $\Gamma(S)$ does not contain any triangles.
 
 Suppose $x \in \nn(S)$ and $x^n=0$, $x^{n-1}\ne 0$, $n\geq 3$. Thus, $x-x^{n-1}$ is an edge in 
  $\Gamma(S)$. Note that $deg(x)=1$ since otherwise $xy=0$ implies that $x-y-x^{n-1}-x$ is a 
  triangle in  $\Gamma(S)$.
  
  In  $\Gamma(S)$ there exists a 4-cycle $a-b-c-d-a$ and since $\diam{\Gamma(S)} \leq 3$, it 
  follows that $x^{n-1} \in \{a,b,c,d\}$. Say, $x^{n-1}=d$. Then $(bx)x^{n-1}=(bx)a=(bx)c=0$ and $bx \ne 0$.
  Since $\Gamma(S)$ does not contain any triangles, $bx=x^{n-1}$. Similarly, $bx^{n-1}=x^{n-1}$.
  Now, $x^{n-1}=bx^{n-1}=bxx^{n-2}=x^{2n-3}=0$, which is a contradiction. It follows that $n= 2$.
\end{proof}

\bigskip
\bigskip


\section{Commutative semirings having zero--divisor graphs with one 3-cycle}

\bigskip

We now proceed to a description of  all graphs with their girths equal to 3, with an additional assumption that they contain 
exactly one 3-cycle. 

The main purpose of the last two sections is to obtain the characterization of all rings (or equivalently all additively
cancellative semirings) having the zero--divisor graph with one 3-cycle, which is a step towards the characterization of rings
with the girth of their zero--divisor graph equal to 3.

\bigskip

\begin{Proposition}\label{thm:characterizationg=3}
    If $S$ is a commutative semiring and $\Gamma(S)$ contains exactly one 3-cycle, then
     $\Gamma(S)=K^{\;  \triangle(r_1,r_2,r_3)}_{m,n}$.
\end{Proposition}

\medskip

\begin{proof}
 If $\Gamma(S)$ contains an $n$-cycle for some $n \ge 4$ then it also contains $C_{4,3}$, 
  $a-b-c-d-a-e-d-a$ as an
 induced  subgraph by Lemma \ref{thm:girth=3}.

 We proceed by adding arbitrary vertices from $\Gamma(S)$ to this subgraph, while showing that in the process we 
 always maintain the structure of  
 $\Gamma(S)=\overline{K}^{\; \triangle(\rho_1,\rho_2,\rho_3)}_{\mu,\nu}$ for some $\rho,\mu,\nu$.
 
 Assume that in $\Gamma(S)$, we have an induced subgraph $\overline{K}^{\; \triangle(\rho_1,\rho_2,\rho_3)}_{\mu,\nu}$,
 $\mu, \nu \geq 2$. 
 Let us decompose the vertex set of  $\overline{K}^{\triangle(\rho_1,\rho_2,\rho_3)}_{\mu,\nu}$ into 
 4 sets: $V_1=\{v; \; \deg{v}=1\}$ (possibly empty), $V_2$ and $V_3$ are the bipartite parts of
 $K_{\mu,\nu}$, where each  vertex in $V_3$ has degree $\mu$, and $V_4=\{e\}$, the top of the
 3-cycle.
 
 Choose any vertex $x \in \Gamma(S)$, that is not in $ \overline{K}^{\triangle(\rho_1,\rho_2,\rho_3)}_{\mu,\nu}$ and
 add it to the graph.
 \begin{itemize}
  \item If $x-v$ is an edge for some $v \in V_1$, then by Theorem \ref{thm:diam3}, $x-w$ is an edge 
  for some $w \in V_2\cup V_3$. If $w \in V_2$, then $\Gamma(S)$ contains $C_{4,5}$ as an induced
  subgraph, and if $w \in V_3$, then $\Gamma(S)$ contains $C_{4,4}$ as an induced
  subgraph. Both conclusions contradict Lemma \ref{thm:induced=3}.
 \item  If $\deg(x)=1$ and $x$ is a neighbour of $a$, $d$ or $e$,  then we get 
        $\overline{K}^{\triangle(\rho_1+1,\rho_2,\rho_3)}_{\mu,\nu}$,  
        $\overline{K}^{\triangle(\rho_1,\rho_2+1,\rho_3)}_{\mu,\nu}$ or  
        $\overline{K}^{\triangle(\rho_1,\rho_2,\rho_3+1)}_{\mu,\nu}$, respectively.
 \item  If $\deg(x)=1$ and $x$ is not a neighbour of $a$, $d$ and $e$, let us assume without loss of generality
    that   $x-v$ is an edge for some $v \in V_2 \backslash \{a,d\}$. Since $\deg(x)=1$, then $xa \ne 0$
    and $(xa)y=(xa)v=(xa)e=0$ for all $y \in V_3$, and this contradicts the assumption
   that  $\Gamma(S)$ has exactly one 3-cycle. 
  \item  If $\deg(x)\geq 2$ and $w-x-v$ is a path, then $v,w \in V_2$ or $v,w \in V_3$. (Otherwise, 
  we obtain a new 3-cycle in  $\Gamma(S)$ if $w\in V_2$ and $v\in V_3$ or if $w$ and $v$  are two vertices of
  the 3-cycle $a-e-d-a$, and we obtain $C_{4,4}$ if one of $v,w$ is equal to $e$.) 
   Say, $v,w \in V_2$.  Suppose there exists $u \in V_2$ such that $xu \ne 0$.
   Now, $(xu)y=(xu)w=(xu)v=0$ for all $y \in V_3$, and this contradicts the assumption
   that  $\Gamma(S)$ has exactly one 3-cycle. 
   Thus, $x-u$ is an edge in $\Gamma(S)$ for all $u \in V_2$, so we get
   $ \overline{K}^{\triangle(\rho_1,\rho_2,\rho_3)}_{\mu,\nu+1}$.
   Similarly, if $v,w \in V_3$, we  get $ \overline{K}^{\triangle(\rho_1,\rho_2,\rho_3)}_{\mu+1,\nu}$. 
 \end{itemize} 
 If the only cycle in $\Gamma(S)$ is the 3-cycle, then all other vertices in $\Gamma(S)$ are at distance
 1 from the triangle. Otherwise, if $a-b-e-a-x-y$ is a subgraph of $\Gamma(S)$, but then
 $xb \ne 0$ and $(xb)a=(xb)e=(xb)y=0$ which is a contradiction, since we obtain a new 3-cycle in 
 $\Gamma(S)$.
\end{proof}

\bigskip

\begin{Corollary}\label{cor:triangle}
    If $S$ is a commutative semiring with the only  cycle of $\Gamma(S)$ being a  3-cycle,
    then $\Gamma(S)=K^{\;  \triangle(r_1,r_2,r_3)}_{1,1}$.
\end{Corollary}

\bigskip

 The following example shows that there exist commutative semirings with their zero--divisor graphs
 equal to   $K^{\;  \triangle(r_1,r_2,r_3)}_{1,1}$ for all $r_1,r_2,r_3 \geq 1$.

\bigskip

 \begin{Example}\label{ex:overlineKtriangle0}
     Let us denote by $e_i$ the element in $\BB^{r_1+r_2+r_3}$, 
     which has its only nonzero entry in the $i$-th position and $f_j$ the element in $\BB^{r_1+r_2+r_3}$, 
     which has its only zero entry in the $j$-th position. 
     Let us define elements 
     \begin{center}
     \begin{tabular}{lll}
      $a_i=f_i \sum\limits_{t=1}^{r_1-1} e_{t} + e_{r_1+r_2} + e_{r_1+r_2+r_3}$ & for all & $i=1,2,\ldots,r_1-1$,\\ 
      $b_j=f_{r_1+j} \sum\limits_{t=1}^{r_2-1} e_{r_1+t}+e_{r_1}+e_{r_1+r_2+r_3}$ & for all & $j=1,2,\ldots,r_2-1$,\\ 
      $c_\ell=f_{r_1+r_2+\ell} \sum\limits_{t=1}^{r_3-1} e_{r_1+r_2+t} +e_{r_1}+e_{r_1+r_2}$ & for all & 
          $\ell=1,2,\ldots,r_3-1$.\\
    \end{tabular}
    \end{center}
    Denote by $S$ the semiring, generated by 
    \begin{eqnarray*}
     {\mathcal Z}&=&\{e_{r_1},e_{r_1+r_2},e_{r_1+r_2+r_3},
       a_i,b_j,c_\ell;  \\
    && \quad=1,2,\ldots,r_1-1, j=1,2,\ldots,r_2-1, \ell=1,2,\ldots,r_3-1\}
    \end{eqnarray*}
    and note that $Z(S)^*={\mathcal Z}$.
    Clearly,  
        \begin{center}
     \begin{tabular}{lcl}
        $e_{r_1}-a_i$ &\quad& $e_{r_1}-(e_{r_1+ r_2}+e_{r_1+ r_2+r_3})$\\
        $e_{r_1+r_2}-b_j$ &\quad& $e_{r_1+r_2}-(e_{r_1}+e_{r_1+ r_2+r_3})$\\
        $e_{r_1+r_2+r_3}-c_\ell$ &\quad& $e_{r_1+r_2+r_3}-(e_{r_1}+e_{r_1+ r_2})$
       \end{tabular}
    \end{center}
    are edges in $\Gamma(S)$ for all 
    $i=1,2,\ldots,r_1-1$, $j=1,2,\ldots,r_2-1$ and  $\ell=1,2,\ldots,r_3-1$, and 
    $e_{r_1}-e_{r_1+r_2}-e_{r_1+ r_2+r_3}-e_{r_1}$ form a 3-cycle.
    Thus, $\Gamma(S)=K^{\;  \triangle(r_1,r_2,r_3)}_{1,1}$.
      \hfill$\square$
 \end{Example}

\bigskip

Recall that we proved in Lemma \ref{thm:girth=3} that all zero--divisor graphs containing exactly one 3-cycle and    at least one  $k$-cycle, $k \geq 4$, also contain $C_{4,3}$ as an induced subgraph. 
The following technical lemma  will give us some algebraic properties on the elements, corresponding to the
vertices of $C_{4,3}$. It will enable us to prove that in this case $\Gamma(S)=K^{\;  \triangle(r_1,r_2,0)}_{m,n}$
where   $r_1,r_2 \geq 1$, $m,n\geq 2$.

\bigskip

\begin{Lemma}\label{thm:technical}
    Suppose $S$ is a commutative semiring,  $\Gamma(S)$ contains exactly one 3-cycle and 
    at least one  $k$-cycle, $k \geq 4$. Let us denote by
       $a-b-c-d-a-e-b-a$   its induced subgraph and let $a-f$ be an edge in $\Gamma(S)$ and 
       $\deg{f}=1$. Then,
    \begin{enumerate}
     \item $a^2=b^2=e^2=0$,
     \item $ac=ec=fc=a$,
     \item $bd=ed=b$,
     \item $a+b=e$.
     \end{enumerate}
 Moreover, if $S$ is additively cancellative, then 
    \begin{enumerate}
     \item[(5)] $2a=2b=2e=0$,
     \item[(6)] $b+e=a$ and $a+e=b$.
    \end{enumerate}
\end{Lemma}

\medskip

\begin{proof}
  Firstly, let us note that $ec\ne0$ and $(ec)a=(ec)b=(ec)d=0$, and since $\Gamma(S)$ contains only one 
  3-cycle, we have $ec=a$. Similarly we prove that $ac=fc=a$ and $bd=ed=b$.
  
  Consider now the  element $a^2$. Since $a^2e=a^2b=0$, and $a^2\ne b$ (otherwise $bd=a^2d=0$),
  $a^2\ne e$ (otherwise $be=a^2e=0$),  $a^2\ne a$ (otherwise $a^2=aa=aec=0$), it follows that 
  $a^2=0$. Similarly we prove that $b^2=0$.

  Now, let us observe that  $(a+b)a=(a+b)b=(a+b)e=0$. Note that $a+b \ne 0$ (otherwise, 
  $bd=(a+b)d=0$),  $a+b \ne a$ (otherwise,  $bd=(a+b)d=ad=0$),  $a+b \ne b$ (otherwise,  
  $ac=(a+b)c=bc=0$), and since $\Gamma(S)$ contains only one 3-cycle, we have $a+b=e$
  and therefore also $e^2=a^2+b^2+2ab=0$.
  
  Suppose now, $S$ is additively cancellative. Then, $2a \ne a$ and $(2a)b=(2a)e=(2a)d=0$, thus 
  $2a=0$. Similarly, $2b=0$ and therefore also $2e=2(a+b)=0$.
  Now, it follows that $a=a+b+b=e+b$ and $b=b+a+a=e+a$.
\end{proof}

\bigskip

\begin{Theorem}\label{thm:completecharacterizationg=3}
    If $S$ is a commutative semiring and  $\Gamma(S)$ contains exactly one 3-cycle and at least one 
    $k$-cycle, $k \geq 4$,    then $\Gamma(S)=K^{\;  \triangle(r_1,r_2,0)}_{m,n}$ and 
    $r_1,r_2 \geq 1$, $m,n\geq 2$.
\end{Theorem}

\begin{proof}
 If $\Gamma(S)$ contains a $k$-cycle for some $k \ge 4$ then it also contains $C_{4,3}$, 
  $a-b-c-d-a-e-b-a$ as an induced  subgraph by Lemma \ref{thm:girth=3}. 
By Proposition \ref{thm:characterizationg=3}, $\Gamma(S)=K^{\;  \triangle(r_1,r_2,r_3)}_{m,n}$ and
  let  $a_i, b_j,e_\ell \in S$ such that $\deg(a_i)=\deg(b_j)=\deg(e_\ell)=1$ 
  and $a-a_i$, $b-b_j$, $e-e_\ell$ are edges in $\Gamma(S)$ for  
  $i=1,2,\ldots,r_1$, $j=1,2,\ldots,r_2$ and  $\ell=1,2,\ldots,r_3$.

 Now, let us prove that $r_1\geq 1$. Consider the sum $d+e$ and observe that $d+e\ne 0$. Namely, if 
  $d+e=0$, we would have $ec=(e+d)c=0$. Since $(e+d)a=0$, we have few possibilities for $e+d$.
  Because $e+d\ne a$ (otherwise $db=(d+e)b=ab=0$), $e+d\ne b$ (otherwise $db=(d+e)b=b^2=0$
  by Lemma \ref{thm:technical} (1)),
  $e+d\ne e$ (otherwise $db=(d+e)b=eb=0$) and $c(d+e) \ne 0$ (otherwise $a=ce=c(d+e)=0$), 
  it follows that $d+e=a_i$ for some $i$.
  Similarly, we prove that $e+c=b_j$ for some $j$ and therefore $r_1,r_2 \geq 1$.

%
%
  Assume $r_3 \geq 1$ and consider an element $ae_\ell \ne 0$. 
  We have $(ae_\ell)b=(ae_\ell)e=(ae_\ell)d=(ae_\ell)a_i=0$
  and since  $\Gamma(S)$ contains only one 3-cycle, we have $ae_\ell=a$. Similarly we prove that 
  $be_\ell=b$.  Now, by Lemma \ref{thm:technical} (4) we have that 
  $e=a+b=ae_\ell+be_\ell=(a+b)e_\ell=ee_\ell=0$, which is a contradiction.
  Thus, $r_3=0$ and $\Gamma(S)=K^{\;  \triangle(r_1,r_2,0)}_{m,n}$.
\end{proof}

\bigskip
\begin{Corollary}\label{thm:AC34}
 If $S$ is an additively cancellative commutative semiring and $\Gamma(S)$ contains exactly one 3-cycle and at least one 
 $k$-cycle, $k \geq 4$,  then 
 $\Gamma(S)=K^{\;  \triangle(n-1,m-1,0)}_{m,n}$ with $m,n \geq 2$.
\end{Corollary}

\medskip

\begin{proof}
  By Theorem \ref{thm:completecharacterizationg=3} we know that $m,n\geq 2$, and denote 
  (as in the proof of the same
  theorem) by $a-b-c_{j}-d_{i}-a-e-b-a$ the induced subgraph of $\Gamma(S)$,
  let $d_1, d_2, \ldots, d_{n-1}, b$ and $c_1, c_2, \ldots, c_{m-1}, a$ be the partition of vertices
  of an induced complete bipartite subgraph of $\Gamma(S)$.
  By Lemma \ref{thm:technical} (1)  we have that $e^2=0$  and therefore $(d_i + e)a=0$, 
   $(d_i+e)e\ne 0$, $(d_i+e)b\ne 0$ and $(d_i+e)c_j\ne 0$ for all $i=1,2,\ldots,n-1$.
   Therefore, $\deg(d_i+e)=1$ and $(d_i+e)-a$ is an edge in $\Gamma(S)$. 
   Since $d_i+e \ne d_j+e$ for $i \ne j$, it follows that  $r_1 \geq n-1$.
  Similarly, we can see   $r_2 \geq m-1$.  
  
 Let $a-f$ be an edge in $\Gamma(S)$ and  $\deg{f}=1$. Using Lemma \ref{thm:technical}, observe
  that  $(f+a)a=0$,  $(f+a)c_i=fc+ac_i=a+a=0$ and $(f+a)e\ne 0$ and thus $f+a=d_{j}$ for some $j$.
 Now, $d_j+b=f+a+b=f+e$ and since $(d_j+b)a=(d_j+b)c_i=0$ for all $i$ and $S$ is additively 
 cancellative, it follows that  $d_{k}=d_j+b=f+e$.  By adding $e$ it follows that $d_k+e=f$, which proves that 
 $r_1=n-1$.  Similarly, $r_2=m-1$.
\end{proof}

\bigskip

Every additively cancellative semiring can be embedded into a ring of differences (see for example \cite[Thm.~5.~11]{HW}), 
but in case the zero--divisor graph of the additively cancellative semiring contains exactly one 3-cycle, we can 
prove that the semiring actually has to be a ring.  We will then study the zero--divisor graphs of rings in the next section. 

\bigskip

\begin{Proposition}\label{thm:ACsemi}
If $S$ is a commutative additively cancellative semiring  and $\Gamma(S)$ contains exactly one 3-cycle, 
then  $S$ is a ring.
\end{Proposition}

\medskip
\begin{proof}
  Denote the only 3-cycle in $\Gamma(S)$ by $a-b-e-a$.  Now $(2a)b=(2a)e=0$, so $2a \in Z(S)$ and 
  $2a \in \{0,b,e\}$, since $2a=a$ implies that $a=0$.  Similarly, we can show that $3a \in \{0,b,e\}$
  and $4a \in \{0,b,e\}$. So, either at least one of $2a,3a,4a$ is equal to zero or at least two of 
  $2a,3a,4a$ coincide.  Since $S$ is additively cancellative, it follows that in all cases there exists an
  integer $n > 0$ such that $na=0$.  Similarly, we can also show that $mb=re=0$ for some integers $m, r > 0$.
  This implies that $(nm)a=(mn)b=0$, so $nm \in Z(S)$ and $nm \in \{0, a, b ,e\}$. In each case we get
  that $N=0$ for some integer $N > 0$, so for every $x \in S$ we have $-x = (N-1)x \in S$, therefore 
  $S$ is a ring.
\end{proof}

\bigskip

The following example shows that in case $S$ is not additively cancellative,  the zero--divisor graph
$\Gamma(S)=K^{\;  \triangle(r_1,r_2,0)}_{m,n}$ need not have $r_1=n-1$ and $r_2=m-1$.

\bigskip 
 \begin{Example}\label{ex:Kmnr1r2}
 Let $S \subseteq \mm_2(\BB) \times \mm_2(\BB) \times \mm_2(\BB)$ be a commutative semiring, 
  generated by
         \begin{center}
     \begin{tabular}{lclcl}
           $a=(J_2,0,0)$&\qquad \qquad&           $b=(0,0,J_2)$  &\qquad \qquad&     $c=(I_2+J_2,0,0)$  \\
        $d=(0,0,I_2+J_2)$&&    $f=(J_2,J_2,I_2+J_2)$ &\qquad&$1=(I_2,I_2,I_2)$
         \end{tabular}
    \end{center}
  Observe that $Z(S)^*=\{a,b,c,d,a+b,f,a+d,b+c\}$ 
  and that $\Gamma(S)=K^{\;  \triangle(2,1,0)}_{2,2}$.
  \hfill$\square$
 
 \end{Example}

\bigskip

\bigskip

\section{Commutative rings having zero--divisor graphs with one 3-cycle}

\bigskip

We now know the types of graphs that can appear as the zero--divisor graphs of semirings.  However, the setting appears to be too 
general to allow for a classification of the structure of semirings that have these types of zero divisor graphs.
We will  characterize all commutative rings (with identity), such that their
zero divisor graphs contain exactly one 3-cycle.

For an arbitrary ring $R$, let $T_2(R)$ denote the ring of all 
 matrices  of the form $aI_2+bJ_2$, where $a,b \in R$.
 
 \bigskip

\begin{Lemma}\label{thm:times}
If $S=S_1\times S_2$ is an additively cancellative commutative semiring such that $\Gamma(S)$ 
contains exactly one 3-cycle, then  $S_1,S_2\in \{\ZZ_4,T_2(\ZZ_2)\}$ are rings and
$\Gamma(S)=K^{\;  \triangle(2,2,0)}_{3,3}$.
\end{Lemma}
\medskip
\begin{proof}
Suppose firstly that all 3 vertices $(x_1,y_1)$, $(x_2,y_2)$, $(x_3,y_3)$ on the 3-cycle have one 
component equal to 0. Since $\Gamma(S)$  contains exactly one 3-cycle, it follows that
at least one of $x_1,x_2,x_3$ is nonzero and at least one of $y_1,y_2,y_3$ is nonzero.
So, let us assume that $(a,0)-(b,0)-(0,c)-(a,0)$ is a 3-cycle.
If $a^2=0$, then  $(a,0)-(a,c)-(b,0)-(a,0)$ is another 3-cycle in the graph and if $a^2\ne0$, then  
$(b,0)-(a^2,0)-(0,c)-(b,0)$ is another 3-cycle in the graph, contradiction. 
 
  Let $(a,b)$, $(a_2,b_2)$ and  $(a_3,b_3)$ be the vertices on the 3-cycle and suppose that 
  $a, b \neq 0$.  Since $(a,0)$ and $(0,b)$ are also zero divisors, we have 
  $\{(a_2,b_2), (a_3,b_3)\} = \{(a,0), (0,b)\}$ and $a^2=b^2=0$.
  If there exists $x \in Z(S_1)$, $x\notin \{0,a\}$, then $xy=0$ for some $y \in Z(S_1)^*$, so either
  $(x,0)-(y,0)-(0,b)-(x,0)$ is another 3-cycle (if $x\ne y$) or $(x,0)-(x,b)-(0,b)-(x,0)$ is another 3-cycle (if $x=y$).
  Thus, $Z(S_1)=\{0,a\}$ and $Z(S_2)=\{0,b\}$. Since $2a\in Z(S_1)$ and $S_1$ is additively 
  cancellative, it follows that $2a=0$.   
  Now, choose an $x \in S_1\backslash Z(S_1)$. Note that $xa \in Z(S_1)$ implies $xa=a$. 
  Since $(2x)a=0$, it follows that $2x=0$ or $2x=a$. 
  Also, $(x+1)a=0$, so either $x+1=0$ or $x+1=a$.  By adding either $x$ or $x+a$ to these equations, 
  we can conclude that $x=1$ or $x=1+a$. Thus we proved that $S_1=\{0,1,a,1+a\}$. Since 
  $a^2=2a=0$ and we either have $1+1=0$ or $1+1=a$, it follows that either $S_1\simeq T_2(\ZZ_2)$ 
  (via mapping $1\mapsto I_2$, $a \mapsto E_{1,2}$) or 
  $S_1\simeq \ZZ_4$.   Similarly,  we show that  $S_2\in \{\ZZ_4,T_2(\ZZ_2)\}$.  
\end{proof}

\bigskip

In the case $S$ is a ring, 
the following Proposition shows that the assumption that $S$ is a direct product is actually superfluous.

\bigskip

\begin{Proposition}\label{thm:ring34}
Let $R$ be a commutative ring with identity such that $\Gamma(R)$ contains exactly one 3-cycle and at least one 
 $k$-cycle, $k \geq 4$,  then   $R$ is isomorphic to a direct product $R_1\times R_2$, where 
 $R_1,R_2\in \{\ZZ_4,T_2(\ZZ_2)\}$ and  $\Gamma(R)=K^{\;  \triangle(2,2,0)}_{3,3}$.
\end{Proposition}

\medskip

\begin{proof}
  By  Corollary \ref{thm:AC34} it follows that $\Gamma(R)=K^{\;  \triangle(n-1,m-1,0)}_{m,n}$
  with  $m,n\geq 2$. Denote by $a-b-c_{i}-d_{j}-a-e-b-a$ the induced subgraph of $\Gamma(R)$ where
   $d_1, d_2, \ldots, d_{n-1}, b$ and $c_1, c_2, \ldots, c_{m-1}, a$ is the partition of vertices
  of an induced complete bipartite subgraph of $\Gamma(R)$.
  
  By Lemma \ref{thm:technical} we know that $e c_i=a$ and therefore 
  $(c_i-c_j)e=(c_i-c_j)b=(c_i-c_j)d_k=0$ for all $i,j,k$. Thus $c_i-c_j=a$ for all $i\ne j$ and $m,n \leq 3$.
  
 Observe that $(b+c_i)b=0$ and $c_i^2 \ne 0$, because  $c_i^2=0$ yields $(b+c_i)c_i=0$ and since 
 $b+c_i\ne 0$, $b+c_i\ne b$, $b+c_i \ne c_i$ and $b+c_i \ne d_j$ (otherwise, $ac_i=a(b+c_i)=ad_j=0$), 
 we get a contradiction. Since $c_ic_j =a$ together with Lemma  
 \ref{thm:technical}  implies that $c_ia=c_ic_je=ae=0$, it follows that $c_ic_j=c_k$ for some $k$.
  If $m=2$, then $c_1^2=c_1$. Suppose that $m=3$ and assume without loss of generality that 
  $c_1c_2=c_1$. If $c_2$ is not idempotent, i.e., $c_2^2=c_1$, it follows that $c_1^2=c_1c_2^2=c_1c_2=c_1$
  and therefore $c_1$ is idempotent. We proved that $R$ in all cases contains an idempotent $c$. 
  Then, $1=c+(1-c)$ is an orthogonal decomposition of identity, thus $R\simeq Rc \times R(1-c)$.
  Now, the proposition follows by Lemma \ref{thm:times}.
\end{proof}

\bigskip

It remains for us to investigate the zero--divisor graphs with girth equal to 3, containing exactly one cycle. 

\bigskip

\begin{Lemma}\label{thm:K3}
If $R$ is a commutative ring with identity and $\Gamma(R)=K_3$,
then $R$ is isomorphic to one of the following rings: $T_2(\GF{4})$,
 $ \ZZ_4[x] / (x^2+x+1)$,
 $\ZZ_2[x,y]/(x^2,xy,y^2)$, or $\ZZ_4[x]/(2x,x^2)$.
\end{Lemma}
\medskip
\begin{proof}
  If $\Gamma(R)=K_3$, then let $Z(R)=\{0,a,b,e\}$. Suppose there exists $f \in R$ such that 
  $fa=b$ or $fa=e$. Without loss of generality, we can assume that $fa=b$.
  Then $(f+1)a=b+a$ and $a+b\in Z(R)$.  If $a+b=e$, then $aR=Z(R)$. Otherwise, since 
  $R$ is a ring, $a+b \notin \{a,b\}$, so $a+b=0$. Then $b=-a$, and $(1-f)a=a-b=2a \in Z(R)$.
  Note that $2a=0$ implies that $b=-a=a$ and $2a=a$ implies that $a=0$.
  So, consider the case $2a=b=-a$. Since $a+e\in Z(R)$ and is obviously not equal to $0,a,e$,
  we have  $a+e=b=-a$ and thus $e=-2a=a$, s contradiction. Therefore, $2a=e$ and again 
  $aR=Z(R)$.
  
   Since $R$-module $Ra$ is isomorphic to $_R R/\Ann(a)$,
  and $\Ann(a)=Z(R)$, we have that $|R|=16$. By \cite[Thm.~12]{ragh} it follows that
  $R  \simeq T_2(\GF{4})$,  $R \simeq \ZZ_4[x] / (x^2+x+1)$.

  In the remaining case we have that $fa=a$  therefore $(1-f)a=0$ for all $f\in R\backslash Z(R)$.
  We thus have $1-f \in Z(R)$ and therefore $R=\{0,1,a,b,e,1+a,1+b,1+e\}$.  Since the set of zero 
  divisors is closed under addition, $R$ is a local ring of order 8. By \cite[p.~687]{corbas1}, $R$
  is one of the following:
  \begin{itemize}
   \item $\GF{8}$, which has no nontrivial zero--divisors,
   \item $\ZZ_2[x]/(x^3)$, but $\Gamma(R)=P_3$,
   \item $\ZZ_2[x,y]/(x^2,xy,y^2)$, which gives us $\Gamma(R)=K_3$,   
   \item $\ZZ_4[x]/(2x,x^2-2)$, but $\Gamma(R)=P_3$,
   \item $\ZZ_4[x]/(2x,x^2)$, which gives us $\Gamma(R)=K_3$,
   \item $\ZZ_8$, but $\Gamma(R)=P_3$.
  \end{itemize}
  Thus the Lemma follows.
 \end{proof}

\bigskip

We are now in a position to characterize all rings such that their  zero--divisor graphs contain exactly 
one 3-cycle.

\bigskip

\begin{Theorem}\label{thm:13-cycle}
If $R$ is a commutative ring with identity and $\Gamma(R)$ contains exactly one 3-cycle,
 then exactly one of the following statements holds:
\begin{enumerate}
  \item
  $R$ is isomorphic to a direct product $R_1\times R_2$, where 
 $R_1,R_2\in \{\ZZ_4,T_2(\ZZ_2)\}$ and $\Gamma(R)=K^{\;  \triangle(2,2,0)}_{3,3}$,
  
  \item
  $R \simeq T_2(\GF{4}))$, 
 $R \simeq  \ZZ_4[x] / (x^2+x+1)$,
 $R \simeq \ZZ_2[x,y]/(x^2,xy,y^2)$, or $R \simeq \ZZ_4[x]/(2x,x^2)$,
 and $\Gamma(R)=K_3=K^{\;  \triangle(0,0,0)}_{1,1}$,


  \item
    $R \simeq \ZZ_{16}$ or $R \simeq \ZZ_2[x]/(x^4)$ and $\Gamma(R)=K^{\;  \triangle(4,0,0)}_{1,1}$.
\end{enumerate}

\end{Theorem}
\medskip
\begin{proof}
If $\Gamma(R)$ apart from the 3-cycle also contains an $n$-cycle for some $n \geq 4$,
then Proposition \ref{thm:ring34} implies (1) and if  $\Gamma(R)=K_3$, then Lemma \ref{thm:K3}
implies (2).

By Proposition \ref{thm:characterizationg=3} the only remaining case is $\Gamma(R)=K^{\;  \triangle(r_1,r_2,r_3)}_{1,1}$, and let  $a-b-e-a$ denote the 3-cycle and let  $a_i, b_j,e_\ell \in S$ such that $\deg(a_i)=\deg(b_j)=\deg(e_\ell)=1$ 
  and $a-a_i$, $b-b_j$, $e-e_\ell$ are edges in $\Gamma(S)$ for  
  $i=1,2,\ldots,r_1$, $j=1,2,\ldots,r_2$ and  $\ell=1,2,\ldots,r_3$.  Note that
  $((1+b_j)a)b= ((1+b_j)a)e=((1+b_j)a)a_i=0$.  This yields $(1+b_j)a=0$, since
$(1+b_j)a=a$ implies that $b_ja=0$. 

If $r_1, r_2, r_3 > 0$ then note that $b_je=e$ and then $(1+b_j)e=e+e=2e=0$, since $(2e)a=(2e)b=(2e)e_\ell=0$
and $2e \neq e$. Since $(1+b_j) b \neq 0$, we have $1+b_j=b$.  By multiplying this equation with $b$, we obtain $b^2=b$ and 
this gives us a contradiction by Lemma \ref{thm:times}.
Therefore $r_3=0$.  

Now, we shall prove that also $r_2=0$.
Since the left $R$-module $Re$ is isomorphic to the quotient module $R/Ann(e)$, and
both $Re$ and $Ann(e)$ have at most 4 elements ($0, e, a, b$), we know that $R$ is a ring of at
most 16 elements.  We also know that $R$ is a directly indecomposable ring by Lemma \ref{thm:times},
therefore it contains no non-trivial idempotents.  Thus, by \cite[Theorem VII.7]{mcdonald} $R$ is a local ring
and the set of zero divisors $Z(R)$ is the Jacobson radical of $R$. 
Assume that $r_2 > 0$. Similarly as above, we can see that $(1+b_j)a \neq a$, so $(1+b_j)a=0$ and thus $1+b_j$ is a zero divisor.
Since $b_j$ is a zero divisor as well, we have that $1=1+b_j+(-b_j)$ is a zero divisor, which gives us a contradiction.

Therefore, $\Gamma(R)=K^{\;  \triangle(r_1,0,0)}_{1,1}$ and we can assume that $r_1 > 0$.  Note that 
$(a^2)a_i=(a^2)b=(a^2)e=0$, so $a^2=0$ since there are no non-trivial idempotents in $R$.  Then 
for each $a_i$ such that $a_ia=0$ we also obtain $(a_i+a)a=(a_i+b)a=(a_i+e)a=0$.
Observe that $a_i+a,a_i+b,a_i+e\notin \{0,a,e,b\}$: for example, if $a_i+b=a$ then $a_ie=0$; 
if $a_i+b=e$ then since $(a+b)a=(a+b)e=0$, it follows that $a+b=e$, and therefore $a_i=a$.  Similarly,
we treat other cases.

Therefore $r_1 \geq 4$ and thus
$\vert Z(R) \vert \geq 8$.  Since $Z(R) \neq R$ this implies that $\vert R \vert = 16$.  Now, 
all rings of order 16 are listed in \cite[pages 687--690]{corbas1}, and we can check which ones are commutative rings
such that their zero--divisor graph only has one 3-cycle.
Among the rings of characteristic 2, the only suitable ring is $\ZZ_2[x]/(x^4)$, since all 3 commutative rings in 
\cite[page 687, case 1.2]{corbas1} have $J^2=\{0,a\}$ and the generators $x_1, x_2, a$ of $J$ give us an additional
3-cycle, either $a-x_1-x_2-a$ in case $x_1x_2=0$, or $a-a_i-a+a_i-a$ in case $x_1^2=x_2^2=0$.  Similarly, we can deal with the
\cite[page 689, case 2.2.a]{corbas1}, which proves that there are no such rings of characteristic 4.  It is easy to check all 
the remaining cases to see that the only other ring that can occur is the ring $\ZZ_{16}$.
\end{proof}

\bigskip
\bigskip


\begin{thebibliography}{References}





  \bibitem{akbari04}
   S.~Akbari, A.~Mohammadian: \emph{On the zero--divisor graph of a commutative ring}, J.~Algebra~ 274 (2004), no. 2, 847--855.

  \bibitem{akbari06}
   S.~Akbari, A.~Mohammadian: \emph{Zero--divisor graphs of non-commutative rings},  J.~Algebra~ 296 (2006),  no. 2, 462--479. 

  \bibitem{akbari07}
   S.~Akbari, A.~Mohammadian: \emph{On zero--divisor graphs of finite rings}, 
   J.~Algebra~ 314 (2007), no.1, 168--184.
 
  \bibitem{and08}
  D.~F.~Anderson: \emph{On the diameter and girth of a zero--divisor graph. II}, Houston~J.~Math.~ 34 (2008),  no. 2, 361--371.
  
  \bibitem{andbad08}
  D.~F.~Anderson, A.~Badawi: \emph{On the zero--divisor graph of a ring}, Comm.~Algebra~ 36 (2008),  no. 8, 3073--3092. 
  
  \bibitem{andlevsha03}
   D.~F.~Anderson, R.~Levy, J.~Shapiro: \emph{Zero--divisor graphs, von Neumann regular rings, and Boolean 
   algebras}, J.~Pure Appl.~Alg.~ 180 (2003), 221--241.

  \bibitem{andliv99}
   D.~F.~Anderson, P.~S.~Livingston: \emph{The zero--divisor graph of a commutative ring}, J.~Algebra~ 217 (1999), 434--447.

  \bibitem{andmul07}
   D.~F.~Anderson, S.~B.~Mulay: \emph{On the diameter and girth of a zero--divisor graph}, J.~Pure Appl.~Alg.~ 210 (2007), 
   543--550.

  \bibitem{atani08}
   S.~E.~Atani: \emph{The zero--divisor graph with respect to ideals of a commutative semiring}, 
   Glas.~Mat.~ 43(63) (2008), 309-–320.

  \bibitem{atani09}
   S.~E.~Atani: \emph{An ideal-based zero--divisor graph of a commutative semiring}, 
   Glas.~Mat.~ 44(64) (2009), 141-–153.
   
  \bibitem{beck88}
  I.~Beck: \emph{Coloring of commutative rings}, J.~Algebra~ 116 (1988), 208--226.

  \bibitem{bozpet09}
  I.~Bo\v zi\' c, Z.~Petrovi\' c: \emph{Zero--divisor graphs of matrices over commutative rings}, Comm.~Algebra~ 37 (2009), 
  no. 4, 1186--1192.

  \bibitem{cann05}
  A.~Cannon, K.~Neuerburg, S.~P.~Redmond: \emph{Zero--divisor graphs of nearrings and semigroups}, 
  in: H. Kiechle, A. Kreuzer, M.J. Thomsen (Eds.), Nearrings and Nearfields, Springer, Dordrecht, The Netherlands, 2005, 
  189--200.

  \bibitem{chleewang10} 
 H.~J.~Chiang-Hsieh, P.~F.~Lee, H.~J.~Wang: \emph{
The embedding of line graphs associated to the zero--divisor graphs of commutative rings},
Israel J. Math. 180 (2010), 193--222. 

  \bibitem{corbas1}
  B.~Corbas, G.~D.~Williams: \emph{Rings of order $p^5$. I. Nonlocal rings.},  J.~Algebra~ 231 (2000),  no. 2, 677--690.


  \bibitem{dem02}
  F.~R.~DeMeyer, T.~McKenzie, K.~Schneider: \emph{The zero--divisor graph of a commutative semigroup}, 
  Semigroup Forum, 65 (2002), 206--214.
  
  \bibitem{demdem}
  F.~R.~DeMeyer, L.~DeMeyer: \emph{ Zero divisor graphs of semigroups}, J.~Algebra~283 (2005), no. 1, 190--198.
  
  \bibitem{dem10}
  L.~DeMeyer, L.~Greve, A.~Sabbaghi, J.~Wang: \emph{The zero--divisor graph associated to a 
  semigroup}, Comm. Algebra 38 (2010), no. 9, 3370--3391. 
  
  \bibitem{dem10rm}
 L.~DeMeyer, Y.~Jiang, C.~Loszewski, E.~Purdy: \emph{ Classification of commutative zero-divisor semigroup graphs}, 
 Rocky Mountain J. Math. 40 (2010), no. 5, 1481--1503.
 
  \bibitem{HW}
 U.~Hebisch, H.~J.~Weinert, \emph{Semirings: algebraic theory and applications in computer science}.
 Series in Algebra, 5. World Scientific Publishing Co., Inc., River Edge, NJ, 1998.
 
   \bibitem{lucas06}
   T.~G.~Lucas: \emph{The diameter of a zero divisor graph}, J.~Algebra~ 301 (2006), 174--193.

  \bibitem{mcdonald}
    B.~R.~McDonald: {\sl Finite rings with identity}, Marcel Dekker Inc., New York, 1974.
    
  \bibitem{ragh}
   R.~Raghavendran: \emph{Finite associative rings},  Compositio Math.~ 21 (1969), 195--229.

  \bibitem{redmond02}
   S.~P.~Redmond: \emph{The zero--divisor graph of a non-commutative ring},  Int.~J.~Commut. 
   Rings 1 (2002), no. 4, 
   (1999), 203--211.





%

\end{thebibliography}
\end{document}